\documentclass[12pt]{amsart} % say
\usepackage{amsmath, amssymb, color, fullpage, tikz}
\newtheorem{theorem}{Theorem}[section]
\newtheorem{lemma}[theorem]{Lemma}

\newtheorem{corollary}[theorem]{Corollary}
\newtheorem{proposition}[theorem]{Proposition}
\newtheorem{definition}[theorem]{Definition}
\newtheorem{remark}[theorem]{Remark}
\def\dz{\hbox{$\mathbb Z$}}
\def\dn{\hbox{$\mathbb N$}}

\def\C{\hbox{$\mathcal C$}}

\def\({\left (}
\def\){\right )}

\def\Ra{\Rightarrow}
\def\La{\Leftarrow}
\def\to{\rightarrow}

\newcommand{\im}[1]{\text{Im}\(#1\)}
\def\C{\hbox{$\mathbb C$}}

\newcommand{\ds}{\displaystyle}

\newcommand{\ba}{\begin{align*}}
\newcommand{\eas}{\end{align*}}

\newcommand{\mylabel}[1]{\label{#1}}%\newcommand{\mylabel}[1]{\label{#1}(*#1*)} %

\begin{document}

\title{Products of Nearly Holomorphic Eigenforms}

\author{Jeffrey Beyerl}
\address[Jeffrey Beyerl]{
Department of Mathematical Sciences\\
Clemson University\\
Box 340975 Clemson, SC 29634-0975
}
\email{jbeyerl@clemson.edu}

\author[James]{Kevin James}
\address[Kevin James]{
Department of Mathematical Sciences\\
Clemson University\\
Box 340975 Clemson, SC 29634-0975
}
\email{kevja@clemson.edu}
\urladdr{www.math.clemson.edu/$\sim$kevja}

\author{Catherine Trentacoste}
\address[Catherine Trentacoste]{
Department of Mathematical Sciences\\
Clemson University\\
Box 340975 Clemson, SC 29634-0975
}
\email{ctrenta@clemson.edu}

 \author[Xue]{Hui Xue}
 \address[Hui Xue]{
Department of Mathematical Sciences\\
Clemson University\\
Box 340975 Clemson, SC 29634-0975
}
\email{huixue@clemson.edu}

\begin{abstract}
We prove that the product of two nearly holomorphic Hecke eigenforms is again a Hecke eigenform for only finitely many choices of factors.
\end{abstract}

\maketitle

\section{Introduction}
It is well known that the modular forms of a specific weight for the full modular group form a complex vector space, and the action of the algebra of Hecke operators on these spaces has received much attention.  For instance, we know that there is a basis for such spaces composed entirely of forms called Hecke eigenforms which are eigenvectors for all of the Hecke operators simultaneously.  Since the set of all modular forms (of all weights) for the full modular group can be viewed as a graded complex algebra, it is quite natural to ask if the very special property of being a Hecke eigenform is preserved under multiplication.  This problem was studied independently by Ghate \cite{Ghate00} and Duke \cite{Duke99} and they found that it is indeed quite rare that the product of Hecke eigenforms is again a Hecke eigenform.  In fact, they proved that there are only a finite number of examples of this phenomenon.  Emmons and Lanphier \cite{EmmonsLanphier07} extended these results to an arbitrary number of Hecke eigenforms.    The more general question of preservation of eigenforms through the Rankin-Cohen bracket operator (a bilinear form on the graded algebra of modular forms) was studied by Lanphier and Takloo-Bighash \cite{Lanphier08, Lanphier03} and led to a similar conclusion.  One can see \cite{Shimura76} or \cite{Zagier77} for more on these operators.

The work mentioned above focuses on eigenforms which are ``new'' everywhere.  It seems natural to extend these results to eigenforms which are not new.  In this paper, we consider modular forms which are ``old'' at infinity in the sense that the form comes from a holomorphic form of lower weight.  More precisely, we  show that the product of two nearly holomorphic eigenforms is an eigenform for only a finite list of examples (see Theorem~\ref{MainResult}).  It would also be interesting to consider the analogous question for forms which are old at one or more finite places.

\section{Nearly Holomorphic Modular Forms}

Let $\Gamma=SL_2(\dz)$ be the full modular group and let $M_k(\Gamma)$ represent the space of level $\Gamma$ modular forms of even weight $k$.  Let $f\in M_k(\Gamma)$ and $g\in M_l(\Gamma)$. Throughout $k,l$ will be positive even integers and $r,s$ will be nonnegative integers.

\begin{definition}
We define Maass-Shimura operator $\delta_k$ on $f\in M_k(\Gamma)$ by 
\[\delta_k(f)=\ds\(\frac{1}{2\pi i}\(\frac{k}{2i \im{z}} + \frac{\partial}{\partial z}\)f\)(z).\] 
Write $\delta_k^{(r)}:=\delta_{k+2r-2}\circ \cdots \circ \delta_{k+2} \circ \delta_k$, with $\delta_k^{(0)}=id$. A function of the form $\delta_k^{(r)} (f)$ is called a nearly holomorphic modular form of weight $k+2r$ as in \cite{Lanphier08}.% and \cite{Shimura2007}.

Let $\widetilde{M}_k(\Gamma)$ denote the space generated by nearly holomorphic forms of weight $k$ and level $\Gamma$. 
\end{definition}

  Note that the image of $\delta_k$ is contained in $ \widetilde{M}_{k+2}(\Gamma)$.  Also, the notation $\delta^{(r)}_k(f)$ will only be used when $f$ is in fact a holomorphic modular form.

We define the Hecke operator $T_n: \widetilde{M}_k(\Gamma)\to \widetilde{M}_k(\Gamma)$ following \cite{Lang95}, as
\[\left(T_n\left(f\right)\right)(z) = \ds n^{k-1}\sum\limits_{d|n}d^{-k}\sum\limits_{b=0}^{d-1} f\left(\frac{nz+bd}{d^2}\right).\]

A modular form (or nearly holomorphic modular form) $f\in \widetilde{M}_k(\Gamma)$ is said to be an eigenform if it is an eigenvector for all the Hecke operators $\{T_n\}_{n\in \dn}$.

The Rankin-Cohen bracket operator $[f,g]_j:M_k(\Gamma)\times M_l(\Gamma) \to M_{k+l+2j}(\Gamma) $ is given by
\begin{equation*}
[f,g]_j := \frac{1}{(2\pi i)^j}\sum_{a+b=j} (-1)^a \binom{j+k-1}{b}\binom{j+l-1}{a}f^{(a)}(z) g^{(b)}(z)
\end{equation*}
where $f^{(a)}$ denotes the $a^{th}$ derivative of $f$.

\begin{proposition}\mylabel{product-rule-v2} Let $f\in M_k(\Gamma)$, $g\in M_l(\Gamma)$. Then
\[\delta^{(r)}_k(f)\delta^{(s)}_l(g) = \ds\sum\limits_{j=0}^{s} (-1)^j\binom{s}{j} \delta^{(s-j)}_{k+l+2r+2j}\left(\delta^{(r+j)}_k(f)g\right).\]
\end{proposition}

\begin{proof}
Note that, $\delta_{k+l+2r}\left(\delta^{(r)}_k(f)g\right) = \delta^{(r+1)}_k(f)g +\delta^{(r)}_k(f)\delta_l(g)$, and use induction on $s$.
\end{proof}

Combining the previous proposition and the Rankin-Cohen bracket operator gives us the following expansion of a product of nearly holomorphic modular forms.

\begin{proposition}\mylabel{coef-equation} Let $f\in M_k(\Gamma)$, $g\in M_l(\Gamma)$. Then
$$\delta^{(r)}_k(f)\delta^{(s)}_l(g) = \ds\sum\limits_{j=0}^{r+s}\frac{1}{\binom{k+l+2j-2}{j}}\left(\sum\limits_{m=\max(j-r,0)}^s (-1)^{j+m} \frac{\binom{s}{m}\binom{r+m}{j}\binom{k+r+m-1}{r+m-j}}{\binom{k+l+r+m+j-1}{r+m-j}}\right)\delta^{(r+s-j)}_{k+l+2j}\([f,g]_j(z)\).$$
\end{proposition}

\begin{proof}
Lanphier \cite{Lanphier03} gave the following formula:
$$\delta_k^{(n)} (f(z)) \times g(z) = \sum_{j=0}^n \frac{(-1)^j \binom{n}{j}\binom{k+n-1}{n-j}}{\binom{k+l+2j-2}{j}\binom{k+l+n+j-1}{n-j}}\delta_{k+l+2j}^{(n-j)}\([f,g]_j(z)\).$$
Substituting this into the equation in Proposition \ref{product-rule-v2}, we obtain
\[\delta^{(r)}_k(f)\delta^{(s)}_l(g) = \ds\sum\limits_{m=0}^s (-1)^m\binom{s}{m}\delta^{(s-m)}_{k+l+2r+2m}\left[\ds\sum\limits_{j=0}^{r+m}\ds\frac{(-1)^j\binom{r+m}{j}\binom{k+r+m-1}{r+m-j}}{\binom{k+l+2j-2}{j}\binom{k+l+r+m+j-1}{r+m-j}}\delta^{(r+m-j)}_{k+l+2j}\([f,g]_j(z)\)\right].\]

Rearranging this sum we obtain the proposition.
\end{proof}

%$$\delta^{(r)}_k(f)\delta_l^{(s)}(g)=  \ds\sum\limits_{j=0}^{r+s}\frac{\delta^{(r+s-j)}_{k+l+2j}[f,g]_j(z)}{\binom{k+l+2j-2}{j}}\left(\sum\limits_{m=j-r}^s (-1)^{j+m} \frac{\binom{s}{m}\binom{r+m}{j}\binom{k+r+m-1}{r+m-j}}{\binom{k+l+r+m+j-1}{r+m-j}}\right)$$

We will also use the following proposition which shows how $\delta_k$ and $T_n$ almost commute. 

%@%@%@%%%%%%%%%%%%%%%%%%%%%%%%%%%%%%%%%%%%
\begin{proposition}\mylabel{delta-T-commute}
Let $f\in M_k(\Gamma)$.
% and let $T_n$ denote the $n^{\mbox{th}}$ Heck operator and $\delta_k$ is the Maass-Shirmura operator. 
Then
\[\left(\delta^{(m)}_k\left(T_nf\right)\right)(z) = \ds \frac{1}{n^m}\left(T_n\left(\delta^{(m)}_k(f)\right)\right)(z)\]
where $m\geq 0$.
\end{proposition}

\begin{proof}
Write $F(z)=f\(\ds\frac{nz+bd}{d^2}\)$. Note that $\ds\frac{\partial}{\partial z}\left(F(z)\right)  = \ds\frac{n}{d^2}\frac{\partial f}{\partial z}\left(\frac{nz+bd}{d^2}\right)$, so that 
\begin{align*}\delta_k\left(T_nf\right)(z) &=&& \ds n^{k-1}\sum\limits_{d|n}d^{-k}\sum\limits_{b=0}^{d-1} \left(\frac{1}{2\pi i}\right)\left[\frac{k}{2i\mathrm{Im}(z)}F(z) + \frac{\partial}{\partial z}\left(F(z)\right)\right]  \\ &=&& \ds n^{k-1}\sum\limits_{d|n}d^{-k}\sum\limits_{b=0}^{d-1} \left(\frac{1}{2\pi i}\right)\left[\frac{k}{2i\mathrm{Im}(z)}f\ds\left(\frac{nz+bd}{d^2}\right) + \frac{n}{d^2}\frac{\partial f}{\partial z}\left(\frac{nz+bd}{d^2}\right) \right].
\end{align*}
Next one computes that $$T_n\left(\delta_k (f)\right)(z) = n\left[\ds n^{k-1}\sum\limits_{d|n}d^{-k}\sum\limits_{b=0}^{d-1}\left(\frac{1}{2\pi i}\right)\left(\frac{k}{2i\mathrm{Im}(z)}f\left(\frac{nz+bd}{d^2}\right) + \frac{n}{d^2}\frac{\partial f }{\partial z}\left(\frac{nz+bd}{d^2}\right)\right)\right]$$ from which we see

\[\left(\delta_k\left(T_nf\right)\right)(z) = \ds\frac{1}{n}\left(T_n\left(\delta_k (f)\right)\right)(z).\]

Now induct on $m$.
\end{proof}
%@%@%@%%%%%%%%%%%%%%%%%%%%%%%%%%%%%%%%%%%%

We would like to show that a sum of eigenforms of distinct weight can only be an eigenform if each form has the same set of eigenvalues. In order to prove this, we need to know the relationship between eigenforms and nearly holomorhpic eigenforms. 

%$%$%$%%%%%%%%%%%%%%%%%%%%%%%%%%%%%%%%%%%%
%\begin{proposition}\mylabel{eform-iff-eform}
%Let $\lambda_n$ be the eigenvalue corresponding to $f\in M_k(\Gamma)$ under the Hecke operator $T_n$. Then $\delta_k^{(r)} f$ is an eigenform if and only if $f$ is an eigenform. In this case $n^r\lambda_n$ is the eigenvalue for $\delta^{(r)}_k f$.
%\end{proposition}

\begin{proposition}\mylabel{eform-iff-eform}
 Let $f\in M_k(\Gamma)$.  Then $\delta_k^{(r)} (f)$ is an eigenform for $T_n$ if and only if $f$ is.  In this case, if  $\lambda_n$ denotes the eigenvalue of $T_n$ associated to $f$, then the eigenvalue of $T_n$ associated to  $\delta^{(r)}_k (f)$ is $n^r\lambda_n$.
\end{proposition}

\begin{proof}
Assume $f$ is an eigenform.  So $\left(T_nf\right)(z) = \lambda_n f(z).$ Then applying $\delta_k^{(r)}$ to both sides and applying Proposition \ref{delta-T-commute} we obtain the following:
$$T_n\left(\delta_k^{(r)} (f)\right)(z) = n^r\lambda_n\left(\delta_k^{(r)} (f)\right)(z).$$
So $\delta_k^{(r)} (f)$ is an eigenform.  

Now assume that $\delta_k^{(r)} (f)$ is an eigenform.  Then $T_n\left(\delta_k^{(r)} (f)\right)(z) = \lambda_n\left(\delta_k^{(r)} (f)\right)(z).$ Using Proposition \ref{delta-T-commute}, we obtain
$\delta_k^{(r)}\left(T_nf\right)(z) = \ds \frac{\lambda_n}{n^r}\delta_k^{(r)}(f)(z) = \delta_k^{(r)}\left(\ds\frac{\lambda_n}{n^r}f\right)(z).$
Since $\delta_k^{(r)}$ is injective, 
$$\left(T_nf\right)(z) = \ds\frac{\lambda_n}{n^r}f(z).$$
Hence $f$ is an eigenform.  

\end{proof}
%$%$%$%%%%%%%%%%%%%%%%%%%%%%%%%%%%%%%%%%%%

%~%~%~%%%%%%%%%%%%%%%%%%%%%%%%%%%%%%%%%%%%
%\begin{proposition}\mylabel{thm-eform-sum-finite}
%Let $f_i\in M_{k_i}(\Gamma)$, $k_{i_1}\neq k_{i_2}$ for $i_1\neq i_2$, and $a_i\in \C^*$. Then $\ds\sum_{i=1}^{t}a_i \delta_{k_i}^{\(n-\frac{k_i}{2}\)}(f_i)$ is an eigenform if and only if every $\delta_{k_i}^{\(n-\frac{k_i}{2}\)}f_i$ is an eigenform and they have the same eigenvalues.
%\end{proposition}

Now our result on a sum of eigenforms with distinct weights follow. 

\begin{proposition}\mylabel{thm-eform-sum-finite}
 Suppose that $\{f_i\}_i$ is a collection of modular forms with distinct weights~$k_i$.  Then $\ds\sum_{i=1}^{t}a_i \delta_{k_i}^{\(n-\frac{k_i}{2}\)}(f_i)$ ($a_i\in \C^*$) is an eigenform if and only if every $\delta_{k_i}^{\(n-\frac{k_i}{2}\)}(f_i)$ is an eigenform and each function has the same set of eigenvalues.
\end{proposition}

\begin{proof}
By induction we only need to consider $t=2$. 

$(\La):$ If $T_n\(\delta_k^{(r)} (f)\)=\lambda \delta_k^{(r)} (f)$, and $T_n\( \delta_l^{\(\frac{k-l}{2}+r\)}(g)\) = \lambda \delta_l^{\(\frac{k-l}{2}+r\)}(g)$, then by linearity of $T_n$,
$$T_n\(\delta_k^{(r)} (f) + \delta_l^{\(\frac{k-l}{2}+r\)}(g)\) = \lambda \(\delta_k^{(r)} (f) + \delta_l^{\(\frac{k-l}{2}+r\)}(g)\).$$

$(\Ra):$ Suppose $\delta_k^{(r)} (f) + \delta_l^{\(\frac{k-l}{2}+r\)}(g)$ is an eigenform. Then by Proposition \ref{eform-iff-eform} and linearity of $\delta_k^{(r)}$, $f + \delta_l^{\(\frac{k-l}{2}\)}(g)$ is also an eigenform. Write
$$T_n\(f + \delta_l^{\(\frac{k-l}{2}\)}(g)\)=\lambda_n\(f + \delta_l^{\(\frac{k-l}{2}\)}(g)\).$$

Applying linearity of $T_n$ and Proposition \ref{delta-T-commute} this is 
$$T_n(f) + n^{\frac{k-l}{2}}\delta^{\(\frac{k-l}{2}\)}_{l}\(T_n(g)\)=\lambda_n f + \lambda_n \delta^{\(\frac{k-l}{2}\)}_{l}(g).$$ Rearranging this we get
$$T_n(f) - \lambda_n f = \delta^{\(\frac{k-l}{2}\)}_{l}\(\lambda_n g - n^{\frac{k-l}{2}}T_n (g)\).$$

Now note that the left hand side is holomorphic and of positive weight, and that the right hand side is either nonholomorphic or zero, since the $\delta$ operator sends all nonzero modular forms to so called nearly holomorphic modular forms. Hence both sides must be zero. Thus we have
$$T_n(f)=\lambda_n f \text{ and } T_n(g) =\lambda_n n^{\frac{-(k-l)}{2}} g.$$

Therefore $f$ is an eigenvector for $T_n$ with eigenvalue $\lambda_n$, and $g$ is an eigenvector for $T_n$ with eigenvalue $\lambda_n n^{\frac{-(k-l)}{2}}$. By Proposition \ref{eform-iff-eform} we have that $\delta_l^{\(\frac{k-l}{2}\)}(g)$ is an eigenvector for $T_n$ with eigenvalue $\lambda_n$. Therefore $f$ and $\delta_l^{\(\frac{k-l}{2}\)}(g)$ are eigenvectors for $T_n$ with eigenvalue $\lambda_n$. So  $\delta_k^{(r)} (f)$ and $\delta_l^{\(\frac{k-l}{2}+r\)}(g)$ must have the same eigenvalue with respect to $T_n$ as well. Hence for all $n\in \dn$,  $\delta_k^{(r)} (f)$ and $\delta_l^{\(\frac{k-l}{2}+r\)}(g)$ must be eigenforms with the same eigenvalues.

\end{proof}
%~%~%~%%%%%%%%%%%%%%%%%%%%%%%%%%%%%%%%%%%%

Using the above proposition we can show that when two holomorphic eigenforms of different weights are mapped to the same space of nearly holomorphic modular forms that different eigenvalues are obtained. 

%@%@%@%%%%%%%%%%%%%%%%%%%%%%%%%%%%%%%%%%%%
\begin{lemma}
Let $l<k$ and $f\in M_k(\Gamma), g\in M_l(\Gamma)$ both be eigenforms. Then $\delta_l^{\(\frac{k-l}{2}\)}(g)$ and $f$ do not have the same eigenvalues. 
\end{lemma}

\begin{proof}
Suppose they do have the same eigenvalues. That is, say $g$ has eigenvalues $\lambda_n(g)$, then by Proposition \ref{eform-iff-eform} we are assuming that $f$ has eigenvalues $n^{\frac{k-l}{2}}\lambda_n(g)$. We then have from multiplicity one there are constants $c, c_0$ such that
\begin{align*}
	f(z) &= \sum_{n=1}^{\infty} c n^{\frac{k-l}{2}}\lambda_n(g) q^n + c_0 \\
	& =\frac{1}{(2 \pi i)^{(k-l)/2}} \frac{\partial^{(k-l)/2}}{\partial z^{(k-l)/2}} \sum_{n=1}^{\infty} c \lambda_n(g) q^n + c_0 \\
	& =\frac{1}{(2 \pi i)^{(k-l)/2}} \frac{\partial^{(k-l)/2}}{\partial z^{(k-l)/2}} g(z) + c_0 
\end{align*}
which says that $f$ is a derivative of $g$ plus a possibly zero constant. However, from direct computation, this is not modular. Hence we have a contradiction. 
\end{proof}

We shall need a special case of this lemma.
\begin{corollary}\mylabel{summands-not-eforms}
Let $k>l$ and $f\in M_k(\Gamma), g\in M_l(\Gamma)$. Then $\delta_l^{\(\frac{k-l}{2}+r\)}(g)$ and $\delta_k^{(r)}(f)$ do not have the same eigenvalues. 
\end{corollary}
%@%@%@%%%%%%%%%%%%%%%%%%%%%%%%%%%%%%%%%%%%

%version divergence ends here

From \cite{Lanphier03} we know that for eigenforms $f,g$, that $[f,g]_j$ is a eigenform only finitely many times. Hypothetically, however, it could be zero. In particular by the fact that $[f,g]_j = (-1)^j [g,f]_j$, $f=g$ and $j$ odd gives $[f,g]_j=0$. Hence we need the following lemma, where $E_k$ denotes the weight $k$ Eisenstein series normalized to have constant term $1$.
\begin{lemma}\mylabel{zero-cohen-brackets} Let $\delta_k^{(r)} (f)\in \widetilde{M}_{k+2r}(\Gamma)$, $\delta_l^{(s)} (g)\in \widetilde{M}_{l+2s}(\Gamma)$. In the following cases $[f,g]_j\neq 0$:

Case 1: $f$ a cusp form, $g$ not a cusp form.

Case 2: $f=g=E_k$, $j$ even.

Case 3: $f=E_k, g=E_l$, $k\neq l$.
\end{lemma}

\begin{proof}
Case 1: Write $f=\ds\sum_{j=1}^{\infty}A_j q^j, g=\ds\sum_{j=0}^{\infty}B_j q^j$. Then a direct computation of the $q$-coefficient of $[f,g]_j$ yields $$A_1 B_0 (-1)^j \binom{j+k-1}{j}\neq 0.$$

Case 2: Using the same notation, a direct computation of the $q$ coefficient yields $$A_0 B_1 \binom{j+l-1}{j} + A_1 B_0 \binom{j+k-1}{j}=2 A_0 A_1 \binom{j+k-1}{j} \neq 0.$$

Case 3: This is proven in \cite{Lanphier03} using $L$-series. We provide an elementary proof here. Without loss of generality, let $k>l$. A direct computation of the $q$ coefficient yields $A_0 B_1 \binom{j+l-1}{j} + A_1 B_0 \binom{j+k-1}{j}$. Using the fact that $A_0=B_0=1$, $A_1=k/B_k$, $B_1=l/B_l$, we obtain
$$\frac{-2l}{B_l}\binom{j+k-1}{j}+(-1)^j\frac{-2k}{B_k}\binom{j+l-1}{j}.$$

If $j$ is even, then both of these terms are nonzero and of the same sign. If $j$ is odd, then we note that for $l>4$,
$$\left | \frac{B_k}{k}\binom{j+k-1}{j} \right | = \left |\frac{(j+k-1) \cdots  (k+1)B_k}{j!}\right| > \left |\frac{(j+l-1) \cdots  (l+1)B_l}{j!}\right |=\left | \frac{B_l}{l}\binom{j+l-1}{j}\right |$$
using the fact that $|B_k|>|B_l|$ for $l>4$, $l$ even. For $l=4$, the inequality holds so long as $j>1$. For $j=1$ the above equation simplifies to $|B_k|>|B_l|$ which is true for $(k,l)\neq (8,4)$, with this remaining cases handled individually. For $j=0$, the Rankin-Cohen bracket operator reduces to multiplication.

\end{proof}

We will need the fact that a product is not an eigenform, given in the next lemma.

%*%*%*%%%%%%%%%%%%%%%%%%%%%%%%%%%%%%%%%%%%
\begin{lemma}\mylabel{both-cusps}
Let $\delta_k^{(r)} (f)\in \widetilde{M}_{k+2r}(\Gamma)$, $\delta_l^{(s)} (g)\in \widetilde{M}_{l+2s}(\Gamma)$ both be cuspidal eigenforms. Then $\delta_k^{(r)} (f) \delta_l^{(s)} (g)$ is not an eigenform.
\end{lemma}

\begin{proof}
By Proposition \ref{coef-equation} we may write $\delta_k^{(r)} (f) \delta_l^{(s)} (g)$ as a linear combination of $\delta_{k+l+2j}^{(r+s-j)}\([f,g]_j\)$. Then from \cite{Lanphier03}, $[f,g]_j$ is never an eigenform. Hence by Proposition \ref{eform-iff-eform}, $\delta_{k+l+2j}^{(r+s-j)}\([f,g]_j\)$ is never an eigenform. Finally Proposition \ref{thm-eform-sum-finite} tells us that the sum, and thus $\delta_k^{(r)} (f) \delta_l^{(s)} (g)$ is not an eigenform. 
\end{proof}
%*%*%*%%%%%%%%%%%%%%%%%%%%%%%%%%%%%%%%%%%%

Finally, this last lemma is the driving force in the main result to come: one of the first two terms from Proposition \ref{coef-equation} is nonzero.

\begin{lemma}\mylabel{first-two-terms}
Let $\delta_k^{(r)} (f)\in \widetilde{M}_{k+2r}(\Gamma)$, $\delta_l^{(s)} (g)\in \widetilde{M}_{l+2s}(\Gamma)$ both be eigenforms, but not both cusp forms. Then in the expansion given in Proposition \ref{coef-equation}, either the term including $[f,g]_{r+s}$ is nonzero, or the term including $[f,g]_{r+s-1}$ is nonzero.
\end{lemma}

\begin{proof}There are three cases. 

Case 1: $f=g=E_k$. If $r+s$ is even, then via Lemma \ref{zero-cohen-brackets}, $[f,g]_{r+s}\neq 0$ and it is clear from Proposition \ref{coef-equation} that the coefficient of $[f,g]_{r+s}$ is nonzero so we are done. If $r+s$ is odd, then $[f,g]_{r+s-1}$ is nonzero. Now because $wt(f)=wt(g)$, the coefficient of $[f,g]_{r+s-1}$ is nonzero. This is due to the fact that if it were zero, after simplification we would have $k=-(r+s)+1 \leq 0$, which cannot occur.

Case 2: If $f$ is a cusp form and $g$ is not then by Lemma \ref{zero-cohen-brackets}, $[f,g]_{r+s}$, and thus the term including $[f,g]_{r+s}$ is nonzero. 

Case 3: If $f=E_k$, $g=E_l$, $k\neq l$. Again by Lemma \ref{zero-cohen-brackets}, $[f,g]_{r+s}$, and thus the term including $[f,g]_{r+s}$ is nonzero. 
\end{proof}

\section{Main Result}
Recall that $E_k$ is weight $k$ Eisenstein series, and let $\Delta_k$ be the unique normalized cuspidal form of weight $k$ for $k\in \{12, 16, 18, 20, 22, 26\}$. We have the following theorem.
%&%&%&%%%%%%%%%%%%%%%%%%%%%%%%%%%%%%%%%%%%
\begin{theorem}\mylabel{MainResult}
Let $\delta_k^{(r)} (f)\in \widetilde{M}_{k+2r}(\Gamma)$, $\delta_l^{(s)} (g)\in \widetilde{M}_{l+2s}(\Gamma)$ both be eigenforms. Then $\delta_k^{(r)} (f) \delta_l^{(s)} (g)$ is not a eigenform aside from finitely many exceptions.  In particular $\delta_k^{(r)} (f) \delta_l^{(s)} (g)$ is a eigenform only in the following cases:
\begin{enumerate}
\item The 16 holomorphic cases presented in \cite{Ghate00} and \cite{Duke99}: 
\[\begin{array}{lll} E_4^2=E_8,  & E_4E_6=E_{10}, &  E_6E_8=E_4E_{10}=E_{14}, \end{array}\]
\[\begin{array}{lll}E_4\Delta_{12}=\Delta_{16}, &  E_6\Delta_{12}=\Delta_{18}, & E_4\Delta_{16}=E_8\Delta_{12}=\Delta_{20}, \end{array}\]
\[E_4\Delta_{18}=E_6\Delta_{16}=E_{10}\Delta_{12}=\Delta_{22},\] \[E_4\Delta_{22}=E_6\Delta_{20}=E_8\Delta_{18}=E_{10}\Delta_{12}=E_{14}\Delta_{12}=\Delta_{26}.\]

%$E_4^2=E_8, E_4E_6=E_{10}, E_6E_8=E_4E_{10}=E_{14}$, $E_4\Delta_{12}=\Delta_{16}$, $E_6\Delta_{12}=\Delta_{18}$, $E_4\Delta_{16}=E_8\Delta_{12}=\Delta_{20}$, $E_4\Delta_{18}=E_6\Delta_{16}=E_{10}\Delta_{12}=\Delta_{22}$, $E_4\Delta_{22}=E_6\Delta_{20}=E_8\Delta_{18}=E_{10}\Delta_{12}=E_{14}\Delta_{12}=\Delta_{26}$.
\item $\delta_4 \(E_4\) \cdot E_4 =\frac{1}{2} \delta_8 \(E_8\)$
\end{enumerate}
\end{theorem}

\begin{proof}
By Proposition \ref{coef-equation} we may write
$$\delta_{k}^{(r)} (f) \delta_{l}^{(s)} (g) = \sum_{j=0}^{r+s} \alpha_j \delta_{k+l+2j}^{\(r+s-j\)}\([f,g]_j\).$$

Now, by Proposition \ref{thm-eform-sum-finite} this sum is an eigenform if and only if every summand is an eigenform with a single common eigenvalue or is zero. Note that by Corollary \ref{summands-not-eforms}, $\alpha_j \delta_{k+l+2j}^{\(r+s-j\)}\([f,g]_j\)$ are always of different eigenvalues for different $j$. Hence for $\delta_{k}^{(r)} (f) \delta_{l}^{(s)}(g)$ to be an eigenform, all but one term in the summation must be zero and the remaining term must be an eigenform.

If both $f,g$ are cusp forms, apply Lemma \ref{both-cusps}. Otherwise from Lemma \ref{first-two-terms} either the term including $[f,g]_{r+s}$ or the term including $[f,g]_{r+s-1}$ is nonzero. By \cite{Lanphier03} this is an eigenform only finitely many times. Hence there are only finitely many $f,g,r,s$ that yield the entire sum, $\delta_k^{(r)} (f) \delta_l^{(s)} (g)$, an eigenform. Each of these finitely many quadruples were enumerated and all eigenforms found. See the following comments for more detail. % \ref{computations} for more detail. 
\end{proof}
%&%&%&%%%%%%%%%%%%%%%%%%%%%%%%%%%%%%%%%%%%

%!%!%!%%%%%%%%%%%%%%%%%%%%%%%%%%%%%%%%%%%%
\begin{remark}In general $2\delta_k \(E_k\) \cdot E_k = \delta_{2k} \(E_k^2\)$. However, for $k\neq 4$, this is not an eigenform.
\end{remark}
%!%!%!%%%%%%%%%%%%%%%%%%%%%%%%%%%%%%%%%%%%

%\begin{remark}\mylabel{computations}
Once we know that $\delta_k^{(r)} (f) \delta_l^{(s)} (g)$ is in general not an eigenform, we have to rule out the last finitely many cases. In particular consider each eigenform (and zero) as leading term $[f,g]_n$ in Proposition \ref{coef-equation}. From \cite{Lanphier03} we know that there are 29 cases with $g$ a cusp form (12 with $n=0$), 81 cases with $f,g$ both Eisenstein series (4 with $n=0$). By case we mean instance of $[f,g]_{n}$ that is an eigenform. We also must consider the infinite class with $f=g=E_k$ and $r+s$ odd, where $[f,g]_{r+s}=0$.

For the infinite class when $f=g$ and $r+s$ is odd we do have $[f,g]_{r+s}=0$. By Lemma \ref{first-two-terms} the $[f,g]_{r+s-1}$ term is nonzero. If $r+s-1=0$, then this is covered in the $n=0$ case. Otherwise $r+s-1 \geq 2$. This is an eigenform only finitely many times. In each of these cases one computes that the $[f,g]_0$ term is nonzero. Thus because there are two nonzero terms, $\delta_k^{(r)} (f) \delta_l^{(s)} (g)$ is not an eigenform.

The 16 cases with $n=0$ are the 16 holomorphic cases. Now consider the rest. In the last finitely many cases we find computationally that there are two nonzero coefficients: the coefficient of $[f,g]_0$, and $[f,g]_{r+s}$. Now $[f,g]_0\neq 0, [f,g]_{r+s}\neq 0$ and so in these cases $\delta_k^{(r)} (f) \delta_l^{(s)} (g)$ is not an eigenform.

The typical case, however, will involve many nonzero terms such as
\begin{align*}\delta_4 \(E_4\) \cdot \delta_4 \(E_4\) &= \frac{-1}{45} [E_4, E_4]_2 + 0\cdot \delta_{10} \([E_4,E_4]_1\) +\frac{10}{45} \delta_{8}^{(2)} \([E_4, E_4]_0\) \\ &= \frac{-1}{45} \(42\cdot E_4 \frac{\partial^2}{\partial z^2} E_4 - 49 \(\frac{\partial}{\partial z}E_4\)^2\) + \frac{10}{45} \delta_8^{(2)} \(E_8\),\end{align*}
$$\delta_6 \(E_6\) \cdot E_8 =\frac{-1}{14} [E_6,E_8]_1 + \frac{3}{7} \delta_{14} \([E_6, E_8]_0\) = \frac{-1}{14} \(6 E_6 \frac{\partial}{\partial z} E_8 -8 E_8 \frac{\partial}{\partial z} E_6\)+ \frac{3}{7} \delta_{14} \(E_6 E_8\) $$
which cannot be eigenforms because of the fact that there are multiple terms of different holomorphic weight.
%\end{remark}

\bibliographystyle{plain}
\bibliography{mybib}
\end{document}